\documentclass[a4paper, 10pt]{amsart}

\usepackage{amsmath, latexsym,  amssymb, amscd}
\usepackage[all]{xy}
\newcommand{\ecuno}{\epsilon_1(\deg V, A^\erre,\eta)}
\newcommand{\ecunoa}{\epsilon_1}
\newcommand{\ecdue}{\epsilon_2(\deg V, A^\g,\eta)}
\newcommand{\ecdueb}{\epsilon_2}

\newcommand{\ef}{\Phi}

\newcommand{\elle}{\mathcal{L}}

\newcommand{\cbh}{{K_0}}

\newcommand{\stabv}{{\rm{Stab}}\,\,V}

\newcommand{\codv}{\cod V}

\newcommand{\cosipa}{}

\newcommand{\alfa}{\alpha}
\newcommand{\berta}{\beta}
\newcommand{\acca}{a}
\newcommand{\bitta}{b}
\newcommand{\Bittaa}{b}

\newcommand{\emme}{m}

\newcommand{\pippo}{\psi}
\newcommand{\fiffo}{\phi}

\newcommand{\enne}{|\tilde\fiffo|}
\newcommand{\ti}{b}

\newcommand{\expm}{{{d}\cosipa+1}}

\newcommand{\coeta}{{}}
\newcommand{\coefa}{{d}\cosipa+\eta}

\newcommand{\czero}{{c_0}}
\newcommand{\cuno}{c_1}

\newcommand{\emor}{\rm End }

\newcommand{\indice}{k}

\newcommand{\rend}{\mathcal{E}}
\newcommand{\erre}{{\underline{r}}}
\newcommand{\esse}{{\underline{s}}}
\newcommand{\g}{{\underline{g}}}

\newcommand{\cod}{\rm cod }

\newcommand{\qe}{\mathbb{Q}}

\newcommand{\re}{\mathbb{R}}
\newcommand{\ze}{\mathbb{Z}}

\newcommand{\Hom}{\mathrm{Hom}}
\newcommand{\upxi}{\zeta}

\newcommand{\varieta}{V}

\newcommand{\punto}{y}

\newcommand{\equattro}{\varepsilon_2}
\newcommand{\ecinque}{\varepsilon_1}

\newcommand{\kzero}{{K_0}}

\newcommand{\kdue}{K_0+||p||}

\newtheorem{con}{Conjecture}[section]
\newtheorem{propo}{Proposition}[section]
\newtheorem{lem}{Lemma}[section]
\newtheorem{cor}{Corollary}[section]

\newtheorem{thm1}[propo]{Theorem}

\newtheorem{thm}{Theorem}[section]
\newtheorem{D}{Definition}[section]

\title[{ The normalized height and non-density statements} ]
{Lower bounds for the normalized height  and non-dense subsets of  varieties in an abelian variety }

\author[{ Evelina Viada}]{  
 }

\begin{document}

\maketitle
\centerline{Evelina Viada\footnote{Evelina Viada,
     Universit\'e de Fribourg Suisse, P\'erolles, D\'epartement de Math\'ematiques, 23 Chemin du Mus\'ee, CH-1700 Fribourg,
Suisse,
    evelina.viada@unifr.ch.}
    \footnote{Supported by the SNF (Swiss National Science Foundation).}
\footnote{Mathematics Subject Classification (2000): 11J95, 14K12, 11G50 and 11D45.\\
Key words: Abelian Varieties, Subvarieties, Heights, Diophantine approximation}}

\begin{abstract}
This work is the third part of a series of papers. In the first two we consider curves and  varieties in a power of an elliptic curve. Here we deal with subvarieties  of an  abelian variety in general. 

Let $V$ be an irreducible  variety of dimension $d$ embedded in an abelian variety $A$, both defined over the algebraic numbers. We say that $V$ is weak-transverse if $V$ is not contained in any proper algebraic subgroup of $A$, and transverse if it is not contained in any translate of such a subgroup.

Assume a conjectural lower bound for the normalized height of   $V$. For $V$ transverse, we  prove that  the  algebraic points  of bounded height of  $V$   which lie in the union of all algebraic subgroups of $A$ of codimension at least  $ d+1$ translated by the  points close to a subgroup $\Gamma$ of finite rank are   non Zariski-dense in  $V$.  If $\Gamma$ has rank zero, it is sufficient to assume that $V$ is  weak-transverse.
 The notion of closeness is defined using a height function.
 
 \end{abstract}

\section{introduction}

All varieties in this article are defined over $\overline{\qe}$. 
Denote by  $A$  a abelian  variety     of
dimension $g$. Consider an irreducible algebraic  subvariety   $V$
of  $A$ of dimension $d$.  
We say that 
\begin{itemize}
 \item $V$ is {\it transverse}, if $V$ is not contained in any
translate of a proper algebraic subgroup of $A$.

\item  $V$ is
{\it weak-transverse}, if $V$ is not contained in  any proper algebraic
subgroup of $A$.
\end{itemize}
As we are going to consider only  algebraic points, we denote by $A=A(\overline\qe)$ and $V=V(\overline{\qe})$.
For a subset $S$ of $A$, we denote by $\overline{S}$ its Zariski closure.
Given  a subset $V^e$ of $V$, an integer $ \indice$ with $1 \le \indice \le g$ and a subset $F$ of $A $, we define the set
\begin{equation}
\label{skv}
S_{\indice}(V^e,F)=V^e \cap \bigcup_{\mathrm{cod}B \ge \indice} B+F, 
\end{equation}
where $B$ varies over all abelian  subvarieties of $A$ of
codimension at least $\indice$ and $$B+F=\{b+f \,\,\,: \,\,\,b\in B, \,\,\,f\in F\}.$$ We  denote the set $S_{\indice}(V^e, A_{\rm Tor})$
simply  by $S_{\indice}(V^e)$, where $A_{\rm Tor}$ is the torsion of $A$.

Nowadays a vast number of theorems and conjectures claim the non-density of sets of the type (\ref{skv}).
 Among others, we recall the  Manin-Mumford, Mordell, 
Mordell-Lang,  Bogomolov and Zilber-Pink Conjectures.  For more literature one can look in the references  of  \cite{io} or \cite{irmn}.

Consider on $A$  a semi-norm $||\cdot||$ induced by a   height function. 
For $\varepsilon \ge 0$, we define $$\mathcal{O}_{\varepsilon}=\{ \xi \in A  : ||\xi|| \le \varepsilon\}$$ 
(Note that in the literature, often, the notation $\mathcal{O}_\varepsilon$ corresponds to the set, we denote in this work, $\mathcal{O}_{\varepsilon^2}$).
Let $\Gamma$  be a subgroup of finite rank in  $A $.  We denote
$\Gamma_{\varepsilon}=\Gamma + \mathcal{O}_\varepsilon.$

Following the
 school  of Bombieri, Masser and Zannier \cite{BMZ}, \cite{BMZ1}, \cite{BMZ2}, one can state the following:
 \begin{con}
\label{gen}
There exists  $\varepsilon>0$ such that: 

\begin{enumerate}

\item  If  $V$ is weak-transverse, then  $S_{d+1}(V,\mathcal{O}_\varepsilon)$  is non Zariski-dense in $V$.

\item  If $V$ is transverse, then  $S_{d+1}(V,\Gamma_\varepsilon)$ is non Zariski-dense in $V$.
\end{enumerate}
\end{con}
For $\varepsilon=0$ this conjecture gives  special cases of the Zilber-Pink Conjecture.
In view of several works, at present,  it is clear that such a conjecture can be split  in two parts; one for the height and the other for the non-density property.

 \begin{con}[Bounded Height Conjecture]
\label{hb}
There exists  $\varepsilon>0$
  and a non empty Zariski-open set $V^0 \subset V$ such that:

\begin{enumerate}
\item  If $V$ is weak-transverse, then   $S_{d+1}(V^0,\mathcal{O}_\varepsilon)$  has bounded height.

\item  If $V$ is transverse, then   $S_{d+1}(V^0,\Gamma_\varepsilon)$ has bounded height.
\end{enumerate}
\end{con}
For $\cbh\ge0$, we denote
$$V_\cbh=V \cap \mathcal{O}_\cbh.$$

 \begin{con}[Non-density Conjecture]
\label{nd}

For all reals $\cbh$, there exists  an effective $\varepsilon>0$ such that: 
\begin{enumerate}

\item  If  $V$ is weak-transverse, then $S_{d+1}(V_\cbh,\mathcal{O}_\varepsilon)$  is non Zariski-dense in $V$.

\item  If $V$ is transverse, then   $S_{d+1}(V_\cbh,\Gamma_\varepsilon)$ is non Zariski-dense in $V$.
\end{enumerate}
\end{con}
In the first instance (see section \ref{equivalenza}), we prove:
\begin{thm}
\label{equii}
Conjecture \ref{nd} i. and ii. are equivalent.
\end{thm}
That i. implies ii. is quite elementary. The other implication is delicate.
It is worth to note that, on the contrary, Conjecture \ref{hb} i. and ii. are not equivalent. It is true that i. implies ii., but the reverse does not hold in general.

In their work, Bombieri, Masser and Zannier, introduce the toric analogue of Conjecture \ref{gen} ii. for  $\Gamma=0$ and $\varepsilon=0$. They present a method to tackle the non-density question  based on the use of Siegel's Lemma and of  the Generalized Lehmer Conjecture (see \cite{fra1}). 

In our previous works \cite{io} and \cite{irmn} we present a different  method for varieties in a power of an elliptic curve. Our method avoids Siegel's Lemma and  the Generalized Lehmer Conjecture. We use instead Dirichlet's Theorem and an Effective Bogomolov Conjecture.  Here we extend our method to sub-varieties of  abelian vaireties in general.

Assume that $V$ is not a union of translates of abelian subvarieties. The   Bogomolov conjecture, 
nowadays a theorem of Zhang  \cite{zang}, claims
$$
\mu(V)=\inf\{\varepsilon>0,\quad \overline{V_\varepsilon}=V\}>0\;,
$$
Since $S_{d+1}(V_\cbh,\mathcal{O}_\varepsilon) \supset S_{g}(V_\cbh,\mathcal{O}_\varepsilon)=V\cap \mathcal{O}_\varepsilon$, Conjecture \ref{nd} i. implies an effective Bogomolov Conjecture, for weak-transverse subvarieties of $A$. Similarly,  Conjecture \ref{nd} ii. implies the Mordell-Lang plus Bogomolov Theorem \cite{poonen}. With effective, we mean that an explicit value for $\varepsilon$ can be given.

We are going  to prove a  strong reverse implication; an Effective  Bogomolov Conjecture for transverse varieties implies Conjecture \ref{nd}. We need a lower bound for $\mu(V)$,  for $V$ transverse, which is functorial with 
respect to the choice of the polarization on $A$. We state a weak form  of \cite{spec} Conjecture 1.5 part ii. for varieties in abelian varieties.

 \begin{con}[Functorial Bogomolov type bound]
\label{bofu1}
Let $(A,\mathcal{L})$ be a polarized  abelian variety of dimension ${{g}}$ defined over a number field $\mathbb{K}$ of degree $[\mathbb{K}:\qe]$.  Let $X$ be a transverse subvariety of $A$ with finite stabilizer. 
Let $\psi:A \to A$ be an isogeny. Then
$$\mu_{\psi^*\mathcal{L}}(X)> c({{g}},[\mathbb{K}:\qe], h_{\mathcal{L}}(A)) \min_{\eta'=\pm \eta}\left(\frac{\deg_{\psi^*\mathcal{L}}A}{\deg_{\psi^*\mathcal{L}}X}\right)^{\frac{1}{2\cod X}+\eta'}$$
where $ c({{g}},[\mathbb{K}:\qe], h_{\mathcal{L}}(A))$ is a constant depending only on ${{g}}$,  $[\mathbb{K}:\qe]$ and $h_\elle(A)$.
\end{con}
This lower bound is expected to hold for all polarizations and  for varieties which are not a union of translated of algebraic subgroups. We prefer to assume as little as possible. We only need the bound for transverse varieties. Furthermore, the assumption on the  stabilizer could be deleted all 
along the article, if we suppose that Conjecture \ref{bofu1} holds also for varieties with stabilizer of positive dimension.

In section \ref{dimmain}, we prove:
\begin{thm}
\label{main}
If Conjecture \ref{bofu1} holds for $V$ then Conjecture \ref{nd} holds for $V$.

\end{thm}

Even if the theorem is conjectural, it is nice to see that the codimension of the algebraic subgroups is the optimal $d+1$. No other known methods, even conjectural (for example assuming the generalized Lehmer's Conjecture) give such an optimal result, for $\varepsilon>0$.
It is also worth to note that we do not need to assume Conjecture \ref{bofu1} for all varieties, but only for the variety $V$ in question. This is an advantage with respect to other methods.

The strategy to prove  Theorem \ref{main}  is based on two steps.
A union of infinitely many sets is non Zariski-dense if and only if:
\begin{itemize}
 \item[(1)] the  union can be taken over    finitely
many sets, and
\item[(2)] all sets in the union are
non Zariski-dense.
\end{itemize}

The proof of (1)  is a 
typical problem of Diophantine approximation. We approximate an
algebraic subgroup with a subgroup of degree bounded by a constant.  This part is an extension of the method introduced in \cite{io} where the ambient variety is a power of an elliptic curve $E^g$.   The ring of endomorphisms of an abelian variety can be quite more complicated than the one of $E$. This produces some problems to overcome. 

The second step (2) is a problem of   height theory. Its proof relies on  Conjecture \ref{bofu1}. This approach differs from the one presented in \cite{io} and \cite{irmn}.  There we use a different kind of bound for the normalized height.


The only known effective bound for the essential minimum in an abelian variety in general is 
given by S. David and P. Philippon \cite{sipa} Theorem 1.4.
This bound is not sharp enough to deduce non-density statements, using the method presented in this article.

 \newpage

\section{preliminaries}

\subsection{Notations}

\begin{itemize}
\item
All varieties are defined over $\overline{\qe}$.

\item For $i=0,1,\dots ,n$, the $A_i$ are non isogenous  simple abelian varieties of dimension $d_i$.

\item $\rend_i$ the ring of endomorfisms of $A_i$ and $t_i$ its rank over $\ze$.

\item $\tau^i=(\tau^i_1,\dots, \tau^i_{t_i})$  a set of integral generators  of $\rend_i$ as $\ze$-module.

\item $\underline{1}=(1,\dots,1)$, $\g=(g_1,\dots ,g_n)$, $\erre=(r_1,\dots ,r_n)$ and $\esse=(s_1,\dots ,s_n)$  multi-indices of natural numbers with $g=\sum_ig_i$, $r=\sum_ir_i$ and $s=\sum_is_i$.

\item$A^\g=A_1^{g_1}\times \dots \times A^{g_n}$ the ambient variety of dimension $\sum_id_ig_i$.

\item $\rend=\rend_1\times \dots \times \rend_n$ the ring of endomorfisms of $A^{\underline{1}}$  and $t$ its rank over $\ze$.
\item $\tau=(\tau^1,\dots, \tau^n)$  a set of integral generators  of $\rend_i$ as $\ze$-module.

\item$V$ a proper irreducible algebraic subvariety of $A^\g$ of dimension $d$.

\item $\Gamma$ a submodule of $A^\g $.

\item   $\gamma$ a set of free generators of $\Gamma^{\underline{1}}$ (see relation (\ref{basegamma}) )  satisfying relation (\ref{gammagrande}).

\item$\varepsilon\ge0$ a non negative real (small and usually positive).

\item $\mathcal{O}_\varepsilon$ the set of points of norm at most $\varepsilon$.

\item$K_0\ge0$ a  real (eventually the norm of $S_{2d+1}(V^0, \Gamma_\varepsilon)$ or $S_{2d+1}(V^0\times p, \mathcal{O}_\varepsilon)$ if bounded).

\item $V_{K_0}=V \cap \mathcal{O}_{K_0}$.

\item$B$ a proper algebraic subgroup of $A^\g$.

\item$\phi$ a (weighted) surjective morphism from $A^\g$ to $A^\erre$.

\item $i_\erre:A^\erre \to A^\g$ an immersion, if it exists, such that $\phi\cdot  i_\erre=[a]$ for $a \in \mathbb{N}^*$.

\item  $\tilde\phi=(\phi|\phi')$ a (special) surjective morphism from $A^{\g+\esse}$  to $A^\erre$ with $\phi:A^\g \to A^\erre$ and $\phi':A^\esse \to A^\erre$.

\item  $B_\phi=\ker \phi$.

\item $\phi_B$ a weighted morphism of minimal dimension such that $B\subset \ker \phi_B$.

\item  $\erre$ rank of $B$ or $\phi_B$.

\item  $k$ the codimension of $B$ or $\phi_B$, note that $k=\sum_id_ir_i$.

\item  $| \phi|$ the maximum of the (Rosati)-norm of the entries of $\phi$.



\item  $\xi,\upxi$ points of small height.

\item  $p$ a point in $A^\esse$ of rank $\esse$.

\item We denote by $\ll$ an inequality up to a multiplicative constant depending on irrelevant parameters of the problem.
\end{itemize}

\subsection{The ambient variety}
\label{nota}

In the  first instance we analyse  the ambient variety.  
Statements on boundness of heights and on non-density of sets  are  invariant under an
isogeny of the ambient variety. Namely, given an isogeny $J:A \to A'$ between abelian varieties over $\overline{\mathbb{Q}}$, Conjecture \ref{nd} (\ref{gen} and \ref{hb}) holds for $V\subset A$ if and only if it  holds for $J(\varieta)$ in $A'$. 
We want to fix a convenient isogeny which simplifies the setting.

In view of the decomposition theorem, an abelian variety  is  isogenous  to a product 
$A_1^{g_1} \times \dots \times A_n^{g_n}$ where the $A_i$ are non isogenous simple abelian varieties of dimension $d_i$.

Let $\underline{g}$ be  the multi-index $(g_1,\dots ,g_n)$ and $g=\sum_i g_i$. 
We denote by $$A^{\g}=A_1^{g_1} \times \dots \times A_n^{g_n}.$$
Note that the dimension of $A^\g$ is $\sum _id_ig_i$.

A point $x\in A^{\g}$ has coordinates $(x_1,\dots, x_{g})$. 
We will also denote 
$$x=(x^1_1,\dots, x^1_{g_1},\dots ,x^n_1,\dots,x^n_{g_n})=(x^1, \dots ,x^n)$$
meaning that $x^i_{k}=x_{k+\sum_{j=1}^{i-1}s_j}$ and $x^i\in A_i^{g_i}$.

In the following we will use other multi-indeces:
 $\underline{r}=(r_1,\dots ,r_n)$ with $r=\sum_ir_i$, $\underline{s}=(s_1,\dots ,s_n)$ with $s=\sum_is_i$. Always we will have $$r_i\le g_i.$$

We denote   by $A^{\erre}=A_1^{r_1}\times \dots \times A_n^{r_n}$ and $A^{\esse}=A_1^{s_1} \times \dots \times A_n^{s_n}$.

\subsection{Subgroups}
Let $M$ be a  $R$-module of rank $s$. 
We define a set of free generators of $M$ as  a set of $s$ 
linearly independent  elements of $M $. 
If $M$  is a free  $R$ module of  rank  $s$  
we call a set of $s$ generators of $M$, integral generators of $M$.

Here we will simply say module for a module over the ring of endomorphism of an abelian variety.

Let $\Gamma$ be a subgroup of $A^\g(\overline\qe)$ of finite rank.

Note that any subgroup of finite rank of $A^\g$ is contained in a $\mathcal{E}$-module of finite rank. In turn a $\rend$-module   of finite rank  in $A^\g$  is a subgroup of finite rank .

The $i$-th saturated module  ${\Gamma}_i\subset A_i$ of $\Gamma$ is
$${\Gamma}_i=\{\phi(y)  \in A_i\,\,\,:\,\,\, \phi \in \Hom(A^\g ,A_i){\rm \,\,\,and \,\,\,}Ny\in \Gamma{\rm \,\,\,with\,\,\,}  N\in {\mathbb{N}}^*\}.$$
By $\Gamma^{\erre}$ we denote $\Gamma_1^{r_1}\times \dots \times \Gamma_n^{r_n}$.
Note that, $\Gamma^{\g}\,$ is invariant via the image
or preimage of isogenies of $A^{\g}$. Furthermore it contains $\Gamma$ and it is of finite rank. This shows that  to prove finiteness statements for $\Gamma$, it is enough to prove them for $\Gamma^{\g}$. 

We denote by $s_i$ the rank of $\Gamma_i$. Let $\gamma^i_{1}, \dots ,\gamma^i_{s_i}$ be a set of free generators of $\Gamma_i$.

We denote
\begin{equation}
\label{basegamma}
\begin{split}
\gamma^i&=(\gamma^i_1,\dots, \gamma^i_{s_i}),\\
\gamma&=(\gamma^1, \dots, \gamma^n).
\end{split}
\end{equation}
Then $\gamma$ is a set of free generators of $\Gamma^{\underline{1}}$. 
We will also denote 
$$\gamma=(\gamma_1, \dots ,\gamma_s)$$
meaning that $\gamma^i_{k}=\gamma_{k+\sum_{j=1}^{i-1}s_j}$.
Note that  $\gamma^i \in A^{s_i}_i$ and $\gamma \in A^{\esse}$.

\begin{D}Let $p\in A^{\g}$ be a point. We denote by $$\Gamma_p=\{\phi(p)\,\,\,:\,\,\,  \phi\in \emor (A^\g)\}$$
(where $\phi$ does not need to be surjective). We can then consider the
associated saturated module $\Gamma_p^{\erre}$,  defined as above.

We say that $p$ has rank $\esse$ if $ \Gamma_{p,i}$ has rank $s_i$.
\end{D}

\subsection{ The norm of a Morphism}
The ring of endomorphisms of $A^{\g}$ is far more complicated than the one
of an elliptic curve. However, it is a free $\ze$-module of finite rank.
We denote by $\rend_i$ the ring of endomorphism of $A_i$. This is a free  $\ze$-module of rank $t_i$. We denote by $\tau^i$ a set of $t_i$ integral generators of $\rend_i$. Then,  a
morphism $\phi_i:A_i^{g_i}\to A_i^{r_i}$ is identified with a $r_i\times
g_i$ matrix with entries in $\rend_i$.

The Rosati involution defines a norm $|\cdot|$ on $\rend_i$.   The $\ze$-module $(\rend_i,|\cdot|)$ is a lattice.

For $\phi_i:A^{g_i}_i\to A^{r_i}_i$, we define $| \phi_i|$ as the
maximum of the (Rosati-) norm of its entries.

Note  that we can identify $\rend_i$ either with an
order in a number field or with a
quaternion ring. In an order, the Rosati-norm is identified  with  the
standard Euclidean norm  in $\mathbb{C}$. A
quaternion ring can be identified with a ring of matrix
with entries in an order. Then,  the Rosati-norm of $a$  will be the trace of $a \bar{a}$.

Since the simple factors of $A^\g$ are not isogenous, a morphism $\phi:A^{\g} \to A^{\erre}$ is identified with a block
matrix

\begin{equation*}
\phi=[\phi_1, \dots , \phi_n]:=
\left(
\begin{array}{ccccc}
\phi_1 & 0 & \dots &\empty &0\\
\empty &\ddots& \empty &\empty &\empty \\
\empty   &\empty &\ddots&\empty &\empty \\
0& \dots&0 &\empty &\phi_n
\end{array}
\right)
\end{equation*}
with $\phi_i:A^{g_i}\to A_i^{r_i}$. 
Note that $\ker \phi= \ker \phi_1 \times \dots \times \ker \phi_n$. Furthermore  $| \phi|=\max_i| \phi_i|$. There are only finitely many morphism of norm smaller than a given constant.

We finally remark that in our previous articles \cite{io} and \cite{irmn} we denote $H(\phi)$ what we denote here $|\phi|$.

\subsection{Weighted and Special morphisms}
As in the elliptic case, there are matrices which have certain advantages. We generalize the definitions given in \cite{io} for a power of an elliptic curve.
The following definitions are  less restrictive, in the sense that we allow common factors of the entries and we work up to an absolute constant depending on the endomorphism ring of $A^\g$.
Up to reordering of columns which does not mix the blocks, a weighted matrix looks like
\begin{equation*}
\phi=\left(
\begin{array}{cccccccccc}
a&\dots& 0 &L^1_1 & 0 & \dots& \empty &\dots&0 \\
\empty &\ddots &\empty &\empty &\empty&\empty &\empty &\empty &\empty\\
0&\dots& a&L^1_{r_1} & 0 & \dots& \empty &\dots&0\\
\empty &\empty &\empty &\empty &\ddots &\empty &\empty &\empty &\empty\\
\empty &\empty &\empty &\empty  &\empty &\ddots&\empty &\empty &\empty\\
0 &\empty& \dots& \empty  &0&a&\dots& 0 &L^n_1 \\
\empty &\empty &\empty &\empty  &\empty &\empty &\ddots &\empty \\
0 &\empty &\dots&\empty  &0 &0&\dots& a&L^n_{r_n}  \\
\end{array}
\right)
\end{equation*}

where $L^i_j:A^{g_i-r_i}_i \to A_i$  for $i=1,\dots,n$, and $| \phi|\ll |a| $. If $g_i=r_i$, we simply forget $L^i_j$.

\begin{D} [Weighted Morphisms]

We say that a surjective morphism  $\phi=[\phi^i,\dots,\phi^n]:A^{\g} \to A^{\erre}$ is weighted  if:
\begin{enumerate}
\item  there exists $a\in \mathbb{N}^*$ such that $aI_{r}$ is a submatrix of $\phi$,
\item $| \phi|\ll a$.
\end{enumerate}

We associate to a weighted morphism $\phi$ an embedding $i_\erre:A^\erre \to A^\g$ such that $\phi \cdot i_\erre=[a]$.

\end{D}

\begin{D} [Special Morphisms]

We say that $\tilde\phi=( \phi|\phi'):A^{\g}\times A^{\esse} \to A^{\erre}$ is special if:

\begin{enumerate}
\item  $\phi$ is weighted,
\item $| \tilde\phi|\ll | \phi|$.
\end{enumerate}

\end{D}

(The  multiplicative constants in the previous two definitions depend only on $\rend$ and not on the morphism.)

\subsection{Algebraic Subgroups}
By the decomposition theorem for abelian varieties, we know that an abelian subvariety of $A^\g$ is isogenous to a product $A^\esse$ for some $s_i \le g_i$.
 Masser and W\"ustholz \cite{mw} Lemma 1.2, prove that the
algebraic subgroups of $A^{\g}$ split as product of algebraic subgroups of $A^{g_i}_i$. In fact non-split algebraic subgroups  would define an isogeny between the non isogenous simple factors. Then,

\begin{lem} 
\label{masserw}

 An algebraic subgroup $B$ of $A^{\g}$ is of the form $B_1 \times \dots \times B_n$ for $B_i$ an algebraic subgroup of $A_i^{g_i}$. Furthermore, the codimension of $B_i$ is $d_ir_i$ for integers $0\le r_i\le g_i$. {\rm(Recall that $d_i$ is the dimension of $A_i$)}.

\end{lem}

\begin{D}
Let $B=B_1 \times \dots \times B_n$ be an algebraic  subgroup of $A^\g$. Let $k_i$ be the codimension of $B_i$ in $A_i^{g_i}$. The rank of $B_i$ is $r_i=k_i/d_i$ and the rank of $B$ is $\erre=(r_1,\dots,r_n)$. 

Let $\phi:A^\g \to A^\erre$ be a surjective morphism. The codimension of $\phi$ is $\sum d_ir_i$, in other words it is the codimension of $\ker \phi$.
\end{D}

Lemma \ref{masserw} implies that, as in the case of $E^g$, an algebraic
subgroup $B$ of rank $\erre$ is contained in the kernel of a surjective morphism $\phi_B:A^{\g} \to A^{\erre}$
and the kernel $B_\phi$ of a surjective morphism  $\phi_B:A^{\g} \to A^{\erre}$ is an algebraic subgroup of rank $\erre$. 

Furthermore, the codimension of $B_\phi$ is  given by
$$\cod \,\,B_\phi=\sum_i d_ir_i.$$ 
Also note that $r=\sum_ir_i$ is the rank of $\phi$ as matrix, and  $r_i$ is the rank of $\phi_i$, for $\phi=[\phi_1, \dots ,\phi_n]$.
In a product of elliptic curves, the rank and the codimension of an algebraic subgroup coincide.

\subsection{Relations between weak-transverse and transverse varieties}
We discuss here, how  we   can associate to the couple $(V,\Gamma)$ a weak-transverse variety $\varieta'$, and vice versa.

Let $\varieta$ be   transverse  in $A^{\g}$. Let $\Gamma$ be a subgroup of
finite rank of $A^{\g} $ and $\gamma$ a set of free generators of $\Gamma^{\underline{1}}$. 
 We define
 $$\varieta'=\varieta\times \gamma.$$
  Note that $\varieta'$ is not contained in any proper algebraic subgroup, because the $\gamma_i$ are $\rend_i$-linearly independent and $\varieta$ is transverse.
So $\varieta'$  is weak-transverse  in $A^{\g+\esse}$.

Let $\varieta'$ be weak-transverse in $A^{\underline{n}}$. Let $H_0$ be the abelian
subvariety  of smallest dimension  such that $\varieta'\subset H_0+p$ for
$p\in H_0^\perp$ and $H_0^\perp$ an orthogonal complement of $H_0$. Then $H_0$ is isogneous to $A^\g$ for a multi-index $\g$ and $H_0^\perp$ is isogenous to $A^\esse$ for $\esse={\underline{n}}-\g$.
 We fix an isogeny $$J: A^{\underline{n}} \to H_0\times H_0^{\perp} \to A^\g \times A^\esse,$$ which sends $H_0$ to $A^\g$ and $H_0^\perp $ to $A^\esse$. Then $J(p)\in 0\times A^\esse$. Since $V'$ is weak transverse the projection of $J(p)$ on $A^\esse$ has    rank $\esse$. 

We consider the natural projection
\begin{equation*}
\begin{split}
 \pi:&A^{\g+\esse} \to A^{\g}\\
 &J(\varieta' )\to \pi J(\varieta').
 \end{split}
 \end{equation*}
 
We define
$$\varieta=\pi J(\varieta'),$$
and $$\Gamma=\Gamma_{J(p)}^{\underline{1}}.$$

 Since $H_0$ has minimal dimension, the variety $\varieta$ is transverse in $A^{\g}$.
 
 Note that $$\varieta'=(\varieta\times 0)+J(p).$$

Statements on the boundness of height and on the Zariski non-density of sets are invariant under an isogeny. Then, without loss of generality, we can assume that 
a weak-transverse variety  in $A^{\underline{n}}$ is of the form

$$V\times p $$ with    
\begin{itemize}
\item[-] $V$ a  transverse subvariety of $ A^\g$,
\item[-]  $p$ a point in $ A^\esse$ of rank $\esse$,
\item[-] ${\underline{n}}=\g+\esse$.
\end{itemize}

\subsection{Points of small height}

On each $A_i$, we fix a symmetric ample line bundle $\mathcal{L}_i$. By $\mathcal{L}$ we denote the polarization on a product variety
$A^{\g}$ given as  the tensor product of the
pull-backs of $\mathcal{L}_i$ via the natural projections on the
factors.
 On $A^{\g}$, we consider the height of the maximum defined as 
\begin{equation*}
h(x^1,\dots ,x^n)=\max_{ij}(h_i(x^i_j)),
\end{equation*}
where $h_i(\cdot)$  on $A_i $ is the  canonical N\'eron-Tate height  induced by $\mathcal{L}_i$.
The height $h$ is the square of a norm  $||\cdot||$ on $A^{\g}\otimes \mathbb{R}$.  For a point $x \in A^{\g}$, we write $||x||$ for $||x\otimes1||$.

The height  of a non-empty set $S\subset A^\g$ is the supremum of the heights  of
its elements. The norm of $S$ is  the positive square root
of its height.

For $\varepsilon \ge 0$, we denote 
$$\mathcal{O}_{\varepsilon}=\{ \xi
\in A^{\g}  : ||\xi|| \le \varepsilon\}.$$ 

For a real $K_0\ge 0$ and a subvariety $V$ of $A^\g$, we denote by
$$V_{K_0}= V \cap \mathcal{O}_{K_0}.$$ 
Note that $\mathcal{O}_0=A^{\g}_{\rm{Tor}}$ is the torsion of $A^\g$ and $V_0$ are the torsion points on $V$.
 We define
$$\Gamma_{\varepsilon}=\Gamma + \mathcal{O}_\varepsilon.$$

 Finally, we remark that for any $x\in A^\g$ and any morphism $\phi$,  $$||\phi(x)||\ll | \phi|||x||.$$

\section{The approximation of the morphisms}
\label{subunion}

As for curves, we want to approximate a   morphism with a morphism of  norm bounded by a constant.
  We reduce the problem of approximating  a morphism of abelian varieties, to the approximation of a morphism with entries in $\mathbb{Z}$. 
This is done by considering $\rend$ as a free $\mathbb{Z}$-module.

We recall Dirichlet's Theorem on the rational approximation of reals. 

\begin{thm1}[Dirichlet 1842, see  \cite{S} Theorem 1 p. 24]

\label{diry}
Suppose that $\alpha_1, \dots , \alpha_m$ are $n$ real numbers and that $Q\ge 2$ is an integer. Then there exist integers $\Bittaa, \beta_1,\dots,\beta_m$ with
\begin{equation*}
1\le \Bittaa<Q^m \,\,\,\,{\rm{and}}\,\,\,\,\left|\alpha_i\Bittaa-\beta_i\right|\le \frac{1}{Q}
\end{equation*}
for $1\le i \le m$.

\end{thm1}
Recall that $t$ is the rank of $\rend$ as $\ze$-module and $\tau=(\tau_1,\dots,\tau_t)$ is a set of  integral generators of $\rend$.
Let $e_i$ be the canonical set of integral generators of $\ze^t$. The map $e_i \to \tau_i$ gives an isomorphism of $\ze^t$ and $\rend$. Then, the natural norm  induced by  $\ze^t$ on $\rend$  is equivalent  to the Rosati-norm. This shows:
\begin{lem}
\label{zero}
Given a positive integer $n$,
there exist constants $\czero$ and $\cuno$ depending on $\tau$, $t$  and $n$ such that, for all $\overline\acca \in \rend^n$, $\overline\acca=\alpha_0+ \alpha_1\tau_1+\dots+\alpha_t \tau_t$ with $\alpha_i\in\ze^n$ and $\alpha=(\alpha_1,\dots,\alpha_t)\in \ze^{nt}$, it holds
\begin{equation*}
\czero |\alfa|\le |{\overline{\acca}} |\le \cuno |\alfa|,
\end{equation*}

\end{lem}

We define $$\lambda_\rend=\min_{a\in \rend^*}|a|$$
and \begin{equation}
\label{qzero}
Q_0=2\max\left(1,\frac{1}{\czero}, \frac{\sum_i|\tau_i|}{\lambda_\rend}\right)
\end{equation}
where $c_0$ is as in  Lemma \ref{zero}.

\begin{lem}
\label{ricopri}
Let   $Q\ge Q_0$ be an integer.
Then, for each non trivial   $\,\,{\overline{\acca}} =(\acca_1,\dots , \acca_n)\in \rend^n$ there exists $\Bittaa\in \mathbb{N}^*$ and ${\overline{\bitta}}=(\bitta_1,\dots , \bitta_n)\in \rend^n$ satisfying
\begin{enumerate}
\item $1\le \Bittaa < Q^{nt}$,
\item $\left|{\overline{\bitta}}\right|\ll \Bittaa \ll\left|{\overline{\bitta}}\right|$,
\item $\left|\frac{{\overline{\acca}} }{|{\overline{\acca}} |}-\frac{{\overline{\bitta}}}{\Bittaa}\right| \ll \frac{1}{Q\Bittaa}$.
\end{enumerate}
\end{lem}

\begin{proof}
Let ${\overline{\acca}} =\alfa^1\tau_1+ \dots+ \alfa^t \tau_t$ with $\alfa^i\in \ze^n$. Define $\alfa=(\alfa^1,\dots ,\alfa^t)\in \ze^{nt}$.

The vector $\frac{1}{|{\overline{\acca}} |}\alfa $ belongs to $\re^{nt}$. Applying Dirichlet's Theorem \ref{diry} with $m=nt$ and $(\alpha_1, \dots, \alpha_m)=\frac{1}{|{\overline{\acca}} |}\alfa $, we deduce that there exist an integer $\Bittaa$ and integer vectors  $\berta^1, \dots ,\berta^n \in \ze^n$  such that
\begin{equation}
\label{meno}1\le \Bittaa < Q^{m}\end{equation}
and 
\begin{equation}
\label{uno}\left| \frac{\alfa^i}{|{\overline{\acca}} |}-\frac{\berta^i}{\Bittaa}\right|  \le \frac{1}{Q\Bittaa}.\end{equation}
The relation (\ref{meno}) proves part i.

Define ${\overline{\bitta}}=\sum_i \berta^i\tau_i$ and $\berta=(\berta^1,\dots, \berta^t)$.
By relation (\ref{uno}) and the triangle inequality, \begin{equation}
\label{due2}\left|\frac{{\overline{\acca}} }{|{\overline{\acca}} |}-\frac{{\overline{\bitta}}}{\Bittaa}\right| =\left| \frac{\sum_i\alfa^i\tau_i}{|{\overline{\acca}} |}-\frac{\sum_i\berta^i\tau_i}{\Bittaa}\right| \le \left| \frac{\alfa^i}{|{\overline{\acca}} |}-\frac{\berta^i}{\Bittaa}\right| \sum_j|\tau_j|\le \frac{\sum_j|\tau_j|}{Q\Bittaa}\ll \frac{1}{Q\Bittaa}.\end{equation}
This proves part ii.

From relations (\ref{uno}) and Lemma \ref{zero} we deduce
$$\frac{|\berta^i|}{\Bittaa} \le\frac{1}{Q\Bittaa}+\frac{|\alfa^i|}{|{\overline{\acca}} |} \le \frac{1}{Q\Bittaa}+\frac{|\alfa^i|}{\czero |\alfa|}\le\frac{1}{Q\Bittaa}+\frac{1}{\czero } .$$

Since $Q> 1/\czero$, 
$$|\berta^i|\le \frac{2}{\czero} \Bittaa 
\,\,\,{\rm{and}}
\,\,\,|\berta|\ll \Bittaa.$$
Therefore
$$|{\overline{\bitta}}|\le  |\berta|\sum_i|\tau_i| \ll \Bittaa.$$
This shows the first inequality in part iii.

Let  $l$ be an index such that $|{\overline{\acca}} |=|\acca_l|$.
By relation (\ref{due2}) we have  $$\left| \frac{\acca_l}{|{\overline{\acca}} |}-\frac{\bitta_l}{\Bittaa}\right| \le\frac{\sum_i|\tau_i|}{Q\Bittaa}.$$
Whence
$$\Bittaa= \Bittaa\frac{|\acca_l|}{|{\overline{\acca}} |} \le \frac{\sum_i|\tau_i|}{Q}+|\bitta_l|.$$ Since $Q>\sum_i|\tau_i|/\lambda_\rend$,
$$\Bittaa\ll |\bitta_l|.$$
This shows the second inequality of part iii.

\end{proof}

\begin{lem}
\label{dicov}
Let  $ \phi:A^\g \to A^\erre$  be a weighted morphism and  $i_\erre:A^\erre \to A^\g$  such that $ \phi\cdot i_\erre=[a]$.  Let $n=rg-r^2+1$ and $m=nt$. Let  $Q\ge Q_0$, where $Q_0$ is as in (\ref{qzero}).  Then, there exists a  surjective morphism $ \psi:A^\g \to A^\erre$  satisfying
\begin{enumerate}
\item $1\le\Bittaa< Q^{\emme}$,
\item $| \pippo|\ll \Bittaa,$
\item $\left| \frac{ \pippo}{\Bittaa}- \frac{ \fiffo}{| \phi|} \right|\ll \frac{1}{Q\Bittaa},$
\item $\psi\cdot i_\erre=[\Bittaa]$.
\end{enumerate}
In particular, by ii. and iv., $\psi$ is weighted.

\end{lem}
\begin{proof}
Let
\begin{equation*}
\phi=\left(
\begin{array}{cccccccccc}
a&\dots& 0 &L^1_1 & 0 & \dots& \empty &\dots&0 \\
\empty &\ddots &\empty &\empty &\empty&\empty &\empty &\empty &\empty\\
0&\dots& a&L^1_{r_1} & 0 & \dots& \empty &\dots&0\\
\empty &\empty &\empty &\empty &\ddots &\empty &\empty &\empty &\empty\\
\empty &\empty &\empty &\empty  &\empty &\ddots&\empty &\empty &\empty\\
0 &\empty& \dots& \empty  &0&a&\dots& 0 &L^n_1 \\
\empty &\empty &\empty &\empty  &\empty &\empty &\ddots &\empty \\
0 &\empty &\dots&\empty  &0 &0&\dots& a&L^n_{r_n}  \\
\end{array}
\right)
\end{equation*}

where $L^i_j:A^{g_i-r_i}_i \to A_i$ and \begin{equation}
\label{dia}
|\phi|\ll|a|\ll | \phi|.
\end{equation}

If $| \phi|\le Q^\emme$, no approximation is needed, as $ \phi$ itself satisfies the consequences.

Suppose now that $| \phi|\ge Q^\emme$.
We associate to $\phi$ a vector  $$ \overline{\acca}=(a, L^1_1,\dots, L^1_{r_1}, \dots,  L^n_1,\dots, L^n_{r_n})\in \rend^{rg-r^2+1}.$$ Note that $|\overline{a}|=| \phi|.$
Apply Lemma \ref{ricopri} to the vector $\overline{a}$. Then,  there exists an integer $\Bittaa$ and a vector $\overline{b}$ such that

\begin{itemize}
\item[1)] $1\le \Bittaa < Q^{\emme}$,

\item[2)]   $\left|{{{\overline{\bitta}}}}\right|\ll \Bittaa \ll \left|{{{\overline{\bitta}}}}\right|$
\item[3)] $\left|\frac{{\overline{\acca}} }{|{\overline{\acca}} |}-\frac{{{{\overline{\bitta}}}}}{\Bittaa}\right| \ll \frac{1}{Q\Bittaa}$
\end{itemize}

We reconstruct a matrix $ \psi$ from $\overline{\bitta}$ respecting exactly the same positional rule we used for constructing $\overline{\acca}$ from $ \phi$. 
Namely, let ${\overline{\bitta}}=(b, L'^1_1,\dots, L'^1_{r_1}, \dots,  L'^n_1,\dots, L'^n_{r_n})$, we define
\begin{equation*}
\psi=\left(
\begin{array}{cccccccccc}
b&\dots& 0 &L'^1_1 & 0 & \dots& \empty &\dots&0 \\
\empty &\ddots &\empty &\empty &\empty&\empty &\empty &\empty &\empty\\
0&\dots& b&L'^1_{r_1} & 0 & \dots& \empty &\dots&0\\
\empty &\empty &\empty &\empty &\ddots &\empty &\empty &\empty &\empty\\
\empty &\empty &\empty &\empty  &\empty &\ddots&\empty &\empty &\empty\\
0 &\empty& \dots& \empty  &0&b&\dots& 0 &L'^n_1 \\
\empty &\empty &\empty &\empty  &\empty &\empty &\ddots &\empty \\
0 &\empty &\dots&\empty  &0 &0&\dots& b&L'^n_{r_n}  \\
\end{array}
\right)
\end{equation*}

1) is exactly part i.

2) implies part ii, because $\left|\overline{\bitta}\right|=| \psi|$.

3) gives  part iii.

Part iv. is evident.
\end{proof}

\begin{thm} 
\label{centro}

Let $V\subset A^\g$ be a transverse variety and let $p \in A^{\esse}$ be a point of rank $\esse$. Let $\varepsilon>0$.
There exists  a real $M>0$ such that
 to each special morphism $\tilde\fiffo:A^{\g+\esse}\to A^\erre$  one can associate  a special morphism $\tilde\pippo:A^{\g+\esse}\to A^\erre$  satisfying:
  \begin{enumerate}
 \item $| \tilde\psi|\ll M$,

\item$
\big( (V_{K_0}\times p) \cap (B_{\tilde{\fiffo}} +\mathcal{O}_{\varepsilon/ M  })\big) \subset 
\big((V_{K_0}
 \times p) \cap (B_{\tilde{\pippo}}
  +\mathcal{O}_{\varepsilon'/|\tilde\pippo|  })\big),
$
with  $\varepsilon'\ll \varepsilon$.
\end{enumerate}

\end{thm}
\begin{proof}

Define
\begin{equation*}
\begin{split}
Q&\ge  \max\left(Q_0,\left\lceil\frac{\kdue}{\varepsilon}\right\rceil\right) \,\,\,\,{\rm{where}} \,\,\,Q_0\,\,\,{\rm{is\,\,\,as \,\,\,in\,\,\,(\ref{qzero})}}\\
\emme&= t(r(g+s)-r^2+n)\\
M&=Q^\emme.
\end{split}
\end{equation*}

If $|\tilde\fiffo|\le M$,  we simply define $\tilde\pippo=\tilde\fiffo$. Then $\varepsilon/{M}\le \varepsilon/|\tilde\fiffo|$ and 
$$  (V_{K_0}\times p)\cap \left(B_{\tilde{\fiffo}}+\mathcal{O}_{\varepsilon/{M}  }\right)$$ is contained in the right hand side.

 Now, suppose  that $|\tilde\fiffo|\ge M$. 
 By
 Lemma \ref{dicov} applied with $\fiffo=\tilde\fiffo$,  
there exists an integer $\Bittaa$ and  a matrix $\tilde{\pippo}$  such that  
\begin{itemize}
\item[1)] $1\le\Bittaa< Q^{\emme}=M$.
\item [2)]$| \tilde\pippo|\ll \Bittaa\ll |\tilde\pippo|,$
\item[3)] $\left|\frac{\tilde{\fiffo}}{\enne}-\frac{\tilde{\pippo}}{\Bittaa}\right|\ll \frac{1}{Q\Bittaa}.$
\item[4)] $\tilde\pippo\cdot i_\erre=[\Bittaa]$.\end{itemize}

Since   $\tilde\fiffo$ is special, 2) and 4) imply that  also $\tilde\pippo$ is special, as well.

Let $(x,p) \in V_{K_0}\times p$. We want to show that, if $${\tilde{\fiffo}}((x,p) +\xi)=0$$ for $\xi \in\mathcal{O}_{\varepsilon/{M}  }$, then $${\tilde{\pippo}}((x,p) +\xi')=0$$ for $\xi'\in \mathcal{O}_
  {\varepsilon'/|\tilde\pippo|  }$ and $\varepsilon'\ll \varepsilon$.
 
Let  $\xi''$ be a point in $ A^\erre$ such that $$
[\Bittaa]\xi''=-\tilde\pippo(x,p).$$
Define  $\xi'=i_\erre(\xi'')$. Then $$\tilde\pippo(\xi')=[\Bittaa]\xi''=-\tilde\pippo(x,p)$$ and 
$$\tilde\pippo((x,p)+\xi')=0.$$ It follows
$$(x,p)\in  (V_{K_0}\times p)\cap (B_{\tilde{\pippo}}+\mathcal{O}_{||\xi'||}),$$
where $\tilde\pippo$ is special and $|\tilde\pippo|\ll M$.

It remains to prove that
  $$||\xi'|| \ll \frac{\varepsilon}{|\tilde\pippo|  }.$$  
Obviously
\begin{equation*}
\enne\tilde\pippo(x,p)  = \ti\left({\tilde\fiffo(x,p)}-{\tilde\fiffo(x,p)}\right)+\enne{\tilde\pippo(x,p)}.
\end{equation*}
It holds
\begin{equation*}
\begin{split}||\xi'||=||\xi''||= \frac{||\tilde\pippo(x,p)||}{\Bittaa}&=\frac{1}{\enne\ti}\left|\left|\ti\left({\tilde\fiffo(x,p)}-{\tilde\fiffo(x,p)}\right)+\enne{\tilde\pippo(x,p)}\right|\right|\\
&\le \frac{1}{\enne}\Big|\Big|{\tilde\fiffo(x,p)}\Big|\Big|+\frac{1}{\enne\ti}\Big|\Big|\enne{\tilde\pippo(x,p)}-b{\tilde\fiffo(x,p)}\Big|\Big|.
\end{split}
\end{equation*}
We estimate the two norms on the right. 

On one hand
\begin{equation*}
\frac{||\tilde\fiffo(x,p)||}{\enne}= \frac{||\tilde\fiffo(\xi)||}{\enne} \ll  ||\xi||\\
\le \frac{\varepsilon}{{M}  } \\\le \frac{\varepsilon}{\ti  },
\end{equation*}
where in the last inequality we use that  $\ti\le M$. 

On the other hand, we assumed  $$||(x,p)||\le \kdue.$$

Using relation 3) and that $Q\ge \left\lceil\frac{\kdue}{\varepsilon}\right\rceil$, we estimate
\begin{equation*}
\begin{split}
\frac{1}{\enne\ti}\Big|\Big|\enne{\tilde\pippo(x,p)}-\ti\tilde\fiffo(x,p)\Big|\Big|&\le
\left|\frac{\tilde\fiffo}{\enne}-\frac{\tilde\pippo}{\ti}\right|||(x,p)||\\&\ll
\frac{||(x,p)||}{Q\Bittaa}\\
&\le
\frac{\varepsilon ||(x,p)||}{(\kdue) \ti  }\le
\frac{\varepsilon}{\ti  }.
\end{split}
\end{equation*}

By 2), we conclude

\begin{equation*}
||\xi'||\ll \frac{\varepsilon}{\ti  }+\frac{\varepsilon}{\ti  }\ll \frac{\varepsilon}{|\tilde\pippo|  }.
\end{equation*}

\end{proof}

\section{The non-density of each intersection}

\label{due}

We associate to a surjective morphism $\phi:A^\g \to A^\erre$ an isogeny of $A^\g$.

\begin{D} To a  weighted morphism $\phi=[\phi^1,\dots,\phi^n] :A^\g \to A^\erre$ 
with
$\phi^i=(aI_{r_i}|L^i)$ we associate:
an isogeny $$\Phi=[\Phi^1, \dots, \Phi^n]: A^\g \to A^\g$$
where
\begin{equation*}
\Phi^i=
\left(\begin{array}{c}  
\phi^i\,\,\,\\0\,\,\,\,|\,\, I_{g_i-r_i}
\end{array}\right)=\left(\begin{array}{cc}  
aI_{r_i} &L^i\\0 &\,\,I_{g_i-r_i}
\end{array}\right)
\end{equation*}

\end{D}
We estimate degrees.
\begin{lem} 
\label{df}The following estimates hold:
\begin{enumerate}

\item\begin{equation*}
\deg_\elle  \Phi(V)\ll | \phi|^{2d}\deg_\elle V.\end{equation*}

\item $$ \deg_\elle \phi(V) \ll | \phi|^{2d}\deg_\elle V.$$
\end{enumerate}
\end{lem}

\begin{proof}
In the first instance we show that, for an isogeny $\psi: A^\g \to A^\g$ and a divisor $X\subset A^\g$, 
\begin{equation}
\label{gradoiper}\deg_\elle \psi^{-1} (X) \ll | \psi|^2 \deg_\elle X.
\end{equation}

Let $i_\elle:A^\g \to \mathbb{P}^N$ be the embedding defined by $\elle$. We denote by $z=(z_0: \dots :z_{N})$ the coordinates of $\mathbb{P}^N$.
Recall that the operation of sum in an abelian variety is given by  polynomials in $z$ of degree 2.  Then, $ i_{\elle *}\psi(x)$ are polynomials in $z$  of degree $\ll \max_{ji}|\psi_{ji}|^2=| \psi|^2$, where the multiplicative constant depends on $\g$ and $\deg_\elle A^\g$. So, If $p(z)$ is a polynomial of degree $\deg X$ defining $X$, the variety $\psi^{-1}(X)$ is defined by $p( i_{\elle *}\psi(x))$ which  is a polynomial in $z$ of degree $\ll  | \psi|^{2 }\deg X$.

We now consider a subvariety $V$ of dimension $d$.
We write 
$$\deg_\elle \psi_*(V)=\deg _{\psi^*\elle}V= c_1(\psi^*\elle)^{d} \cdot V,$$
where on the right we mean the intersection number and $c_1(\cdot)$ is a representative of the first Chern-class. By  Bezout's Theorem, 
$$\deg_\elle \psi_*(V)\ll (\deg_\elle c_1(\psi^*\elle))^d \deg_\elle V.$$
Note that $c_1(\psi^*\elle) = \psi^{-1} c_1(\elle)$, and $c_1(\elle)$ is a divisor of degree $\deg_\elle A^\g$.  By (\ref{gradoiper}), we deduce $\deg_\elle c_1(\psi^*\elle) \ll | \psi|^2 \deg_\elle A^\g$. 
We conclude
\begin{equation}
\label{gradoiper1}\deg_\elle \psi_*(V) \ll | \psi|^{2d} \deg_\elle V,\end{equation}

where the multiplicative constant depends on $\g$ and $\deg_\elle A^\g$.

Part i. is given by (\ref{gradoiper1})  applied to $\psi=\Phi$. Note that $\deg_\elle \psi(V)\le \deg_\elle \psi_*(V)$.

 In the chosen polarization, forgetting coordinates makes degrees decrease.
Note that $\phi(V)=\pi\Phi(V)$, where $\pi$ is the projection on  $r$ coordinates.
By part i.
$$\deg_\elle  \phi(V)\le \deg_\elle \Phi(V)\ll  | \phi|^{2d} \deg_\elle V.$$

\end{proof}

\begin{propo}
\label{EM}

Assume that  Conjecture \ref{bofu1} holds.  Let $V\subset A^\g$ be a transverse variety with finite stabilizer. Assume that the codimension of $\phi$ is at least $d+1$. Then, for any $\eta>0$, there exist effective constants $\ecuno$ and $\ecdue$ such that,
 for any point $y\in A^\g $,
 \begin{enumerate}
 \item
$$
 \mu(\phi(V+\punto))>\ecuno \frac{1}{| \phi|^{\coefa}},$$
 \item $$  \mu\left(\Phi(V+\punto)\right )>\ecdue | \phi|^{\frac{1}{\codv}-\eta}.$$
 \end{enumerate}

\end{propo}
\begin{proof}

i.

 Since $V$ is irreducible, transverse and defined over $\overline{\qe}$,  $\phi(V+\punto)$ is as well.

Recall  that   the codimension of $\phi\ge d+1$.  Then $\phi(V+\punto)\subsetneq A^\erre$ has codimension  and dimension at least $1$. 
Apply Conjecture \ref{bofu1} with $A=A^\erre$, $\psi=id_{A^\erre}$ and $X=\phi(V +\punto)$. Then
\begin{equation*}
\mu_\elle(\phi(V+\punto)) > c(A^\erre, \eta) \min_{\eta'=\pm \eta}\left(\frac{\deg _{\elle} A^\erre}{\deg _{\elle} (V+\punto)}\right)^{\frac{1}{2}+\eta'}
\end{equation*}

Degrees are preserved by translations,
hence Proposition \ref{df} ii. implies  $$\deg (\phi(V+\punto))=\deg\phi(V)\ll | \phi|^{2d}\deg V.$$
If follows
$$\mu_\elle(\phi(V+\punto))>c(A^\erre,\eta)
 \frac
 {(\deg A^\erre)^{\frac{1}{2}-\eta}}
 {(\deg V)^{\frac{1}{2}+\eta}}
 \frac{1}{| \phi|^{d+2\eta}}.$$
Define
$$ \ecuno=c\left(A^\erre,\frac{\eta}{2}\right) \frac{(\deg A^\erre)^{\frac{1}{2}-\frac{\eta}{2}}}{(\deg V)^{\frac{1}{2}+\frac{\eta}{2}}}.$$
Then
$$\mu(\phi(V+\punto))>  \frac{\ecuno}{| \phi|^{d+\eta}}.$$

ii. Recall that, for any variety $X$, \begin{equation*}
\begin{split}\mu_{\ef^*\elle}X&=\mu_\elle\left(\Phi(X)\right ),\\
\deg_{\ef^*\elle}X&=\deg_\elle \ef_*X.
\end{split}
\end{equation*}

Apply Conjecture \ref{bofu1} with $A=A^\g$, $\psi=\ef$ and $X=V +\punto$. 
We obtain 
\begin{equation*}
\begin{split}\mu_\elle\left(\Phi(V+\punto)\right )&> c(A^\g, \eta) \min_{\eta'=\pm \eta}\left(\frac{\deg _{\ef^*\elle} A^\g}{\deg _{\ef^*\elle} (V+\punto)}\right)^{\frac{1}{2\cod V}+\eta'}\\
& =c(A^\g, \eta) \min_{\eta'=\pm \eta}\left(\frac{\deg _{\ef^*\elle} A^\g}{\deg _{\ef^*\elle} (V)}\right)^{\frac{1}{2\cod V}+\eta'}.
\end{split}
\end{equation*}

Recall that (see, for instance, \cite{l-b} (6.6) Corollary page 68)  $$\deg_{\ef^*\elle}A^\g=|\ker \ef|=a^{2(\sum_id_ir_i)}.$$
By assumption $\sum_id_ir_i\ge d+1$ and $|\phi|\ll a$. So $$\deg_{\ef^*\elle}A^\g\ge a^{2(d+1)}\gg | \phi|^{2(d+1)}.$$

By Lemma \ref{df} i.,
$$\deg_{\ef^*\elle}(V)=\deg_\elle (\Phi_*(V))\ll | \phi|^{2d} \deg_\elle  V.$$
Thus
\begin{equation*}
\mu_\elle\left(\Phi(V+\punto)\right )> c'(A^\g, \eta)\min_{\eta'=\pm \eta}\left(\frac{| \phi|^{2d+2} \deg_\elle A^\g}{| \phi|^{2d}\deg_\elle V}\right)^{\frac{1}{2\cod V}+\eta'}.
\end{equation*}
Define
$$\ecdue=c'\left(A^\g, \frac{\eta}{2}\right) \min_{\eta'=\pm \frac{\eta}{2}}\left(\frac{\deg_\elle A^\g}{\deg_\elle V}\right)^{\frac{1}{2\cod V}+\eta'}.$$
\end{proof}

We come to the  main proposition of this section; each set in the
union is non Zariski-dense. 
\begin{thm}
\label{finito}
Assume  Conjecture \ref{bofu1}. Let $V\subset A^\g$ be a transverse variety with  finite stabilizer. 
Then, there exists an effective $\ecinque>0$ such that
for $\varepsilon \le \ecinque$, for
 all weighted morphisms $\phi$  of codimension $\ge d+1$ and for all $\punto\in i_\erre(A^\erre)$, the set 
$$\left(V_\kzero+\punto\right)\cap \left(B_\phi+\mathcal{O}_{\varepsilon/| \phi| }\right)$$  
is non Zariski-dense in $V+y$.

\end{thm}
\begin{proof}

Choose 
$\eta\le\frac{1}{2}.$
Let $$\ecunoa=\min_{\erre}\ecuno$$ where  $\ecuno$ is as in Proposition \ref{EM} i. and $\erre$ varies over the multi-indeces such that $\sum d_ir_i\ge d+1$ and $r_i\le g_i$. Let $$\ecdueb=\ecdue$$ be as in Proposition \ref{EM} ii.  
Define  
 \begin{equation*}
 \begin{split}
 m&=\left(\frac{\kzero}{\ecdueb
  }\right)^{\frac{\codv}{1-\coeta(\codv)\eta}},\\
  \ecinque&=\frac{1}{g} \min \left({\kzero}, \frac{\ecunoa}{m^{\expm}}\right).\\
   \end{split}
  \end{equation*}
 Choose $$\varepsilon \le \ecinque.$$

We distinguish two cases: either 
$| \phi| \le m$ or $| \phi| \ge m.$

Case (1)   $\,\,\,| \phi| \le m$.

 Let $x+\punto  \in (V_\kzero+\punto )\cap
(B_\phi+\mathcal{O}_{\varepsilon/| \phi| })$, where $y\in i_\erre(A^\erre)$.   Then  $$\phi(x+\punto )=\phi(\xi)$$
for $||\xi ||\le \varepsilon/| \phi| $. Since   $\varepsilon \le\frac{ \ecunoa }{gm^{\expm}}$ and $| \phi|\le m$, 
$$||\phi(x+\punto )||=||\phi(\xi)||  \le {g\varepsilon}\le{ \frac{ \ecunoa }{m^{\expm}} }\le \frac{ \ecunoa }{| \phi|^{\expm}}.$$
In Proposition \ref{EM} i. we have proven
$${ \frac{ \ecunoa }{| \phi|^{\expm}}}<\mu(\phi(V+\punto)).$$

 We deduce that  $\phi(x+\punto )$ belongs to the non Zariski-dense set $$Z_2=\phi(V+\punto)\cap {\mathcal{O}}_{\epsilon_1/m^\expm  }.$$

Since $V$ is transverse, the dimension of $\phi(V+\punto)$ is at least $1$. Consider  the restriction  morphism $\phi_{|V+\punto}:V+\punto \to \phi(V+\punto)$. Then $x+y$ belongs to the non Zariski-dense set $\phi^{-1}_{|V+\punto}(Z_2)$.

Case (2)  $\,\,\,| \phi|\ge m$.

Let $x +\punto  \in  (V_{\kzero}+\punto )\cap \left(B_\phi+\mathcal{O}_{\varepsilon/| \phi| }\right)$, where $y\in i_\erre(A^\erre)$. Then
$$\phi(x+y)=\phi(\xi)$$  for  $||\xi||\le \varepsilon/| \phi| $ and
$$\Phi(x+y)=(\phi^1(x+y),\overline{x}^1,\dots,\phi^n(x+y),\overline{x}^n),$$  where $\overline{x}^i$ are some of the coordinates of $x$. So
$$ ||\Phi(x+\punto )||\ll  \max\left(||\phi(\xi)||,||x||\right).$$
Since $||\xi||\le \frac{\varepsilon}{| \phi|}$ and
$\varepsilon \le \frac{K_0}{g}$, then $$||\phi(\xi)||\ll  g{\varepsilon}\le K_0.$$
Also $||x||\le K_0$, because $x\in V_\kzero$. Thus
 $$||\Phi(x+\punto )||\le \kzero.$$
 
Since  $| \phi| \ge m=\left(\frac{\kzero}{\ecdueb }\right)^{\frac{\codv}{1-\coeta(\codv)\eta}}$, 
 \begin{equation*}
  \kzero \le\ecdueb | \phi|^{\frac{1}{\codv}-\coeta\eta}.
   \end{equation*}
 In Proposition \ref{EM}  we have proven
 $$\ecdueb
 | \phi|^{\frac{1}{\codv}-\coeta\eta}<\mu(\Phi(V+\punto)).$$  So  
 $$||\Phi(x+\punto )||\le K_0< \mu(\Phi(V+\punto)).$$ 
 We deduce that $\Phi(x+\punto )$ belongs to the  non Zariski-dense set $$Z_1=\Phi(V+\punto)\cap {\mathcal{O}}_{K_0}.$$ 
 The restriction morphism $\Phi_{|V+\punto}:V+\punto \to \Phi(V+\punto)$ is finite, because $\Phi$ is an isogeny. Then 
 $x+y$ belongs to the non Zariski-dense set $\Phi_{|V+\punto}^{-1}(Z_1)$.

\end{proof}

 \begin{propo}
 \label{includo} Let $V\subset A^\g$ be transverse and let $p \in A^\esse$ be a point of rank $\esse$.
Let $\tilde{\phi}=(\phi|\phi'): A^{\g+\esse} \to A^\erre$  be a special morphism. Then, there exists $y\in i_\erre(A^\erre) $ such that  the map $(x,p)\to x+y$ defines an injection
 \begin{equation*}
\left( (V_\kzero\times p) \cap \left(B_{\tilde{\phi}}
 +\mathcal{O}_{\varepsilon/| \phi| }\right) \right)\hookrightarrow \Big(
 (V_\kzero+\punto )\cap
 \left(B_\phi+\mathcal{O}_{\varepsilon'/| \phi| }\right)\Big),
 \end{equation*}
where $\varepsilon'\ll \varepsilon$.

\end{propo}

\begin{proof}

Let $\tilde\phi=(  \phi|\phi')$ be special. By definition of special  $aI_r$ is a submatrix of $\phi$ and $|\phi|\ll a$.  Recall that $\phi \cdot i_\erre=[a]$.

 Let $y'\in A^\erre$ be a point such that 
$$[a]y'=\phi'(p).$$
Define
$$y=i_\erre(y').$$
Then 
\begin{equation}
\label{riccio}
  \phi(\punto)=  [a]\punto'=\phi'(p)
\end{equation}
Let $$(x,p) \in (V_\kzero\times p) \cap
\left(B_{\tilde{\phi}} +\mathcal{O}_{\varepsilon/| \phi|}\right) .$$
 Then, there exists  $\xi \in \mathcal{O}_{\varepsilon/|\phi| }$ such that  $$\tilde\phi((x,p)+\xi)=0.$$ 
Equivalently $$  \phi(x)+\phi'(p)+\tilde\phi(\xi)=0.$$
By relation (\ref{riccio}) we deduce
$$  \phi(x+y)+\tilde\phi(\xi)=0.$$ 

 Let $\xi''\in A^\erre$  be a point such that $$  [a]\xi''=\tilde\phi(\xi).$$
Define
 $\xi'=i_\erre(\xi'')$, then
$$   \phi(\xi')=  [a]\xi''=\tilde\phi(\xi),$$
and $$   \phi(x+\punto+\xi')=0.$$
Since $\tilde\phi$ is special $| \tilde\phi|\ll  | \phi|$. Furthermore $||\xi||\le \frac{\varepsilon}{| \phi| }$. We deduce
$$||\xi'||=||\xi''||=\frac{||\tilde\phi(\xi)||}{  | \phi|}\ll
\frac{\varepsilon}{| \phi| }.$$

 In conclusion  $$(x+\punto ) \in (V_\kzero+\punto )\cap \left(B_{  \phi}+\mathcal{O}_{\varepsilon'/| \phi| }\right) .$$

  \end{proof}

\begin{cor}
\label{fine} Assume Conjecture \ref{bofu1}. Let $V$ be transverse in $A^\g$ and let $p\in A^\esse$ be a point of rank $\esse$. Suppose that $V$ has finite stabilizer. Then,
there exists an effective $\equattro>0$ such that for $\varepsilon \le\equattro$, for all   special morphisms $\tilde{\phi}=(\phi|\phi')$   of codimension  at least $d+1$
 the set 
$$(V_\kzero \times p) \cap \left(B_{\tilde{\phi}}
  +\mathcal{O}_{\varepsilon/| \phi| }\right) $$ is non Zariski-dense in $V\times p$.
\end{cor}
\begin{proof}
This is an immediate consequence of Theorem  \ref{finito} and Proposition \ref{includo}.

\end{proof}

\section{Reduction of the Problem}

\subsection{ Geometry of Numbers}

Let us recall and extend some properties of geomentry of numbers.

\begin{propo}
\label{ret}

Let  $p=(p_1, \dots, p_s)$ be a point of
$A^s_0$ of rank $s$.  
There exist  positive effective constants $c(p,\tau^0)$   and  $ \varepsilon_0(p,\tau^0)$ such that 
$$c(p,\tau^0)\sum_i|b_i|^2||p_i||^2 \le ||\sum_ib_i(p_i-\xi_i)||^2$$
for all  $b_1,\dots,b_s \in \rend_0$  and  for all
$\xi_1,\dots, \xi_s  \in \mathcal{O}_{ \varepsilon_0(p,\tau^0)}$.

\end{propo}

\begin{proof}
The Rosati involution defines a norm on $\rend_0$ which is compatible
with the norm on the abelian variety. Namely
$||b_ip_i||=|b_i|_{\rend_0}||p_i||$. Thus $(\rend_0, |\cdot |)$ is a
hermitian free $\mathbb{Z}$-module of rank $t$ and $(A,||\cdot||)$ is a hermitian $\rend_0$-module.

The proof is then the anologue of the proof of \cite{io} Proposition 3.3  (this proposition is written at the end of the article, after the references), where one shall read $A_0$ instead of $E$ and consider $b=0$.
\end{proof}

\begin{cor}
\label{corret}
Let  $p\in A^s_0$ be  a points of rank $s$.  
Then, there exist  positive effective constants $c(p)$   and  $ \varepsilon_0(p)$ such that, 
for all $\xi \in\mathcal{O}_{ \varepsilon_0(p)}$,
 for all $\phi_0: A_0^s \to A_0$, 
$$ c(p)|\phi_0|\ll ||\phi_0(p-\xi)||.$$
\end{cor}
\begin{proof}
Apply Proposition \ref{ret}, with $(b_1,\dots , b_s)=\phi_0$. 
Note that $\sum_i|b_i|^2||p_i||^2\ge | \phi|^2\min_i||p_i||^2$. Then, the corollary is proven for  $c(p)=c(p,\tau^0)^{\frac{1}{2}}\min_i||p_i||$ and $\varepsilon_0(p)=\varepsilon_0(p,\tau^0)$.
\end{proof}

We now extend this corollary to an abelian variety in general.

\begin{cor}
\label{cor3} 
Let $p \in A^\esse$ be a point of rank $\esse$ and let $\phi:A^\esse \to A^{\underline{1}}$.
Then, there exist positive constants $c(p)$ and $\varepsilon_0(p)$ such that 
$$c(p) | \phi|\ll ||
\phi(p-\xi)||,$$
for all $\xi \in \mathcal{O}_{\varepsilon_0(p)}$ .

\end{cor}
\begin{proof}
 Let  $\phi=[\phi_1,\dots ,\phi_n]$ with $\phi_i:A^{s_i}_i \to A_i$.
Let $p=(p^1,\dots ,p^n)$ with $p^i\in A^{s_i}_i$ of rank $s_i$ and  $\xi=(\xi^1,\dots ,\xi^n)$ with $\xi^i\in A^{s_i}_i$.

Apply Corollary \ref{corret} to each block, with $A_0=A_i$,  $p=p^i$, $s=s_i$, $\phi_0=\phi_i$, $\xi=\xi^i$. Choose  $c(p)$ to be the minimum of $c(p^i)$ and $\varepsilon_0(p)$ to be the minimum of $\varepsilon_0(p^i)$.
\end{proof}

\begin{lem} Given  reals $K_0$, $\varepsilon$ and a subgroup $\Gamma\subset A^\g$ of finite rank, there exists a set of free generators $\gamma=(\gamma_1,\dots, \gamma_s)$ of $\Gamma^{\underline{1}}$  such that,
 for all $\phi: A^\esse \to A^{\underline{1}}$,  it holds
\begin{equation}
\label{gammagrande}
 (K_0+\varepsilon) | \phi|\ll ||\phi(\gamma)||.
\end{equation}

\end{lem}

\begin{proof}
Let $z_1,\dots,z_{s_i}$ be free generators of $\Gamma_i$. Decompose $\Gamma_i\otimes \qe=\qe z_1+\dots \qe \tau_{t_i}z_1+\dots +\qe z_{s_i}+\dots \qe \tau_{t_i}z_{s_i}$. The proof is then an immediate application of Lemma \ref{zero} and  \cite{io}  Lemma 3.4 (this lemma is written at the end of the article, after the references). 

\end{proof}

Given $\Gamma$, We fix, once and for all, a set of free generators $\gamma$ of $\Gamma^{\underline1}$ satisfying   relation (\ref{gammagrande}) for $\varepsilon=K_0$.

  \subsection{Reducing to weighted morphisms}
 Using the  Gauss algorithm we show:
    \begin{lem}
\label{peso1}
Let $\Delta_0\in M_{r_0\times r_0}(\rend_0)$ be a matrix of rank $r_0$. Then, there exists  an integer $a$ and a matrix  $\Delta_0'\in M_{r_0\times r_0}(\rend_0)$ of rank $r_0$ such that $$\Delta_0'\Delta_0=aI_{r_0}.$$
\end{lem}
\begin{proof}

 Note that the $\rend_0$ is not
necessarily commutative, it can be a quaternion, however  given non
zero elements $x,y
\in \rend_0$ there exist $a,b \in \rend_0$ such that $ax=by$. 
This shows that one can operate a Gauss reduction  using only operation on the left and
without commuting elements in $\rend_0$. In other words there exists a matrix $\Delta$ of rank $r_0$ such that $\Delta\Delta_0$ is a diagonal matrix.  Using the norm, we can find a matrix $\Delta'$ of maximal rank $r_0$ such that $\Delta'\Delta\Delta_0=[a_1, \dots, a_{r_0}]$  with $a_i\in \ze^*$.   Let $m$ be the minimum common multiple of $a_1, \dots ,a_{r_0}$.  We define $\Delta'_0= [\frac{m}{|a_1|},\dots , \frac{m}{|a_{r_0}|}]\Delta'\Delta$. 
\end{proof}

This has some immediate consequences.

\begin{lem}
\label{peso}
 Let  $\psi:A^{\g}\to A^{\erre}$ be a surjective morphism. Then, there exists an isogeny $\Delta$ of $A^\erre$ such that $\phi=\Delta \psi$ is a  weighted morphism. As a consequence,
 \begin{enumerate}
 \item $ B_\psi \subset B_\phi+A^{\g}_{\rm{Tor}}.$
 \item For all $\varepsilon\ge0$, $$ \bigcup_{\psi \,\,\,
{\rm{rk}}(\psi)= \erre}  (B_\psi+(\Gamma^\g)_{\varepsilon} )\subset \bigcup_{\substack{\phi\,\,\,\,{\rm{weighted}}\\
{\rm{rk}}(\phi)= \erre}}(B_\phi+(\Gamma^{\g})_{\varepsilon} ).$$
 \end{enumerate}

\end{lem}
\begin{proof}

Let $\psi=[\psi_1,\dots,\psi_n]$.  Let $\Delta_i$ be a submatrix of $\psi_i$ of rank $r_i$ with maximal pivots. By Lemma \ref{peso1} applied to each $\Delta_i$, there exist $\Delta'_i$ such that $\Delta_i'\Delta_i=a_iI_{r_i}$ with  $a_i \in \mathbb{N}^*$. Let $m$ be the mimum common multiple of the $a_i$. Define $$\Delta=\left[\frac{m}{a_1}I_{r_1}, \dots ,\frac{m}{a_n}I_{r_n}\right][\Delta_1',\dots, \Delta_n'].$$
\end{proof}

\subsection{Reducing to special morphisms}

We prove here an important inclusion.
\begin{propo} 
\label{speciali}

To every weighted morphism  $\phi:A^{\g} \to
A^{\erre}$ we can associate a special morphism $
\tilde\phi=( \phi|\phi'):A^{\g+\esse}\to A^{\erre}$ such that, for all $0\le \varepsilon\le K_0$,  the map $x \to (x,\gamma)$ defines  an injection 
$$\big(V_{K_0}\cap (B_\phi+\left(\Gamma^{\g}\right)_{\varepsilon})\big)
\hookrightarrow \left(
(V_{K_0}\times \gamma) \cap (B_{\tilde{\phi}} +\mathcal{O}_{\varepsilon})\right).$$

\end{propo}
\begin{proof}

Recall that $\gamma$ is a set of free generators of $\Gamma^{\underline{1}}$ (see definition (\ref{basegamma})), it has rank $\esse$ and it satisfies relation (\ref{gammagrande}).

Let $x \in V_{K_0}\cap (B_\phi+\left(\Gamma^{\g}\right)_{\varepsilon})$. Then, there exist  points $y\in \Gamma^\g$ 
and $\xi \in \mathcal{O}_{\varepsilon,A^\g}$ such that 
$$\phi(x+y+\xi)=0.$$
As $\gamma$ is a set of free generators,
 there exist  a positive integer $N$ and a morphism
$G:A^\g \to A^\erre$ such that
$$Ny=G\gamma.$$
We define $$\tilde\phi=(N\phi|\phi G).$$

Then
\begin{equation}
\label{tr1}\tilde\phi((x,\gamma)+(\xi,0))=0.
\end{equation}

To prove that  $\tilde\phi$ is special, we shall show that $| \tilde\phi|\ll N| \phi|$, as  we already know that $\phi$ is weighted.
Equivalently, we are going to   show that $|\phi'|\ll N | \tilde\phi|$. Let $l$ and $j$ be indices such  that $|\phi'|=|\phi'_{lj}|$.
Consider the $l$ row of the equation (\ref{tr1}). For $\varphi_l$ and $\varphi'_l$ the $l$-th rows of $\phi$ and $\phi'$ respectively, we have
\begin{equation*}
||N \varphi_l(x+\xi)|| = ||\varphi'_l(\gamma)||.
\end{equation*}

Then 
$$ | N\phi|(||x||+||\xi||)\gg || \varphi_l(x+\xi)|| = ||\varphi'_l(\gamma)||.$$

  By assumption $||{{x}}||\le K_0$ and $||\xi||\le \varepsilon$. So
  $$N | \phi|(K_0+\varepsilon)\gg  ||\varphi'_l(\gamma)||.$$
  
In view of relation (\ref{gammagrande}), we deduce
\begin{equation*}
N | \phi|(K_0+\varepsilon)\gg (K_0+\varepsilon)|\varphi'_l|.
  \end{equation*}

Thus
$$|\varphi'_l|=|\phi'|\ll N | \phi|.$$

\end{proof}

\subsection{Reducing to finite stabilizers}
In the following lemma, we see  that to prove Theorem \ref{main} it is enough to prove it for varieties with finite stabilizer. Recall that $\stabv$ is an algebraic subgroup of $A^\g$.
\begin{lem}
\label{stabi}
\begin{enumerate}

\item
Let $X=X_1\times A^\erre$ be a subvariety of $A^\g$ of dimension d. Then, for $k\ge r$,  $$S_{k}(X,F)\hookrightarrow S_{k-r}(X_1,F')\times A^\erre$$
where $F'$ is the projection of $F$ on $A^{\g-\erre}$.

\item
Let $V$ be a  (weak)-transverse subvariety of   $A^{\g}$.
Suppose that $\stabv $ is isogenous to $A^\erre$ with $\sum r_i\ge1$.
Then, there exists an isogeny $j$ of $A^\g$ such that 
$$j(V) =V_1\times A^{\erre}$$ with $V_1$ (weak)-transverse in $A^{\g - \erre}$ and $\stabv_1$ a finite group.

\item Conjecture \ref{nd} holds if and only if it holds for  varieties with finite stabilizer.
\end{enumerate}

\end{lem}
The proof is the analogue of  \cite{irmn} Lemma 4.1  (this lemma is written at the end of the article, after the references).

\subsection{Reducing to Conjecture \ref{nd} ii}
In this section we are going to  prove Theorem \ref{equii}.
In the first instance, we study some properties of a morphism vanishing on a point of large rank.
\label{equivalenza}

\begin{lem}
\label{lemuccio}
Let $p\in A^\esse$ be a point of rank $\esse$.
Let $\tilde\phi=(\phi|\phi'):A^{\g+\esse}\to A^\erre$ be a surjective morphism. 
Let $\varepsilon\le \varepsilon_0(p)$ where $\varepsilon_0(p)$ is defined as in Corollary \ref{cor3}.

If there exists a point $(x, p) \in B_{\tilde\phi}+\mathcal{O}_\varepsilon$ then
\begin{enumerate}
\item $\phi$ has rank $\erre$,
\item There exists $\tilde\psi=(\psi|\psi'):A^{\g+\esse}\to A^\erre$ with $\psi$ weighted  such that $$B_{\tilde\phi}\subset B_{\tilde\psi}+ {\rm{Tor}}_{A^\g}.$$
\end{enumerate}
\end{lem}

\begin{proof}
i- Suppose that the rank of $\phi$ is less than $\erre$. Then, there exists $\lambda=[\lambda^1,\dots, \lambda^n]$ with $\lambda^i\in \rend_i^{r_i}$ such that 
$$\lambda \phi=0.$$
Let $(x,p)\in  B_{\tilde\phi}+\mathcal{O}_\varepsilon$. Then, there exists $(\xi,\xi') \in \mathcal{O}_\varepsilon$ such that
\begin{equation}
\label{gigio}
\tilde\phi\left((x,p)+(\xi,\xi')\right)=0.
\end{equation}
So
$$\lambda\phi'(p+\xi')=-\lambda \phi(x+\xi)=0.$$
 Corollary \ref{cor3}, applied with $\phi=\lambda \phi'$ and $\xi=\xi'$ implies that $p-\xi$ has rank $\esse$, whence  $\lambda\phi'=0$.
 So $\lambda \tilde\phi=0$. This contradicts that $\tilde\phi$ has full  rank $\erre$.

ii- By part i we can assume that rank $\phi$ is $\erre$.  By Lemma \ref{peso} applied to $\phi$, there exists an invertible $\Delta$ such that $\Delta\phi$ is weighted. Then $\tilde\psi=\Delta \tilde\phi$ satisfies ii.
\end{proof}

We can now prove a statement slightly more precise than Theorem \ref{equii}.

\begin{thm}[Reformulation of Theorem \ref{equii}]
\label{equi34}
Let $\varepsilon\ge0$. Then,

\begin{enumerate}
\item
 The map $x \to (x,\gamma)$ defines an injection
$$S_{\indice}(V,\Gamma_{\varepsilon})\hookrightarrow  S_{\indice}(V\times \gamma,\mathcal{O}_{\varepsilon}).$$ 

\item
Let $p\in A^\esse$ be a point of rank $\esse$. Let 
  $\varepsilon \le \varepsilon_0(p)$, where $\varepsilon_0(p)$  is as in Corollary \ref{cor3}. Then, the map $(x,p) \to x$ defines an injection
$$S_{\indice}(V_{K_0}\times
p,\mathcal{O}_{\varepsilon})\hookrightarrow
S_{\indice}\left(V_{K_0},(\Gamma^\g_{p})_{\varepsilon'}\right),$$
where $\varepsilon'$ depends on $\varepsilon$, $p$,
$\g$ and $K_0$.
\end{enumerate}
\end{thm}

\begin{proof}
 Part i. is an immediate consequence of Proposition \ref{speciali}.

ii) Let $(x,p) \in S_{\indice}(V_{K_0}\times p,\mathcal{O}_{\varepsilon})$.
Then,  there exists a block matrix $\tilde\phi=[\tilde\phi_1,\dots
,\tilde\phi_n]$ of rank $\erre$ with $k\le \sum_id_ir_i$, and $(\zeta,\zeta')\in \mathcal{O}_\varepsilon$ such that 
 \begin{equation}
 \label{tr}\tilde\phi((x,p)+(\zeta,\zeta'))=0.
 \end{equation}
In view of Lemma \ref{lemuccio}, we can assume that $\tilde\phi=(\phi|\phi')$ with $\phi$ weighted. Let $aI_\erre$ be a submatrix of $\phi$ with $|\phi|\ll a$ and $i_\erre:A^\erre \to A^\g$ such that $\phi \cdot i_\erre=[a]$.

 We want to  show that $|\phi'|\ll | \phi|$. Let $l$ and $j$ be indices such  that $|\phi'|=|\phi'_{lj}|$.
Consider the $l$-th row of the equation (\ref{tr}). For $\varphi_l$ and $\varphi'_l$ the $l$-th rows of $\phi$ and $\phi'$ respectively, we have
\begin{equation*}
|| \varphi_l(x+\xi)|| = ||\varphi'_l(p+\xi')||.
\end{equation*}

Then 
$$| \phi|(||x||+||\xi||)\gg ||\varphi_l(x+\xi)|| = ||\varphi'_l(p+\xi')||.$$

  By assumption $||{{x}}||\le K_0$ and $||\xi||\le\varepsilon$. So
  $$ | \phi|(K_0+\varepsilon)\gg  ||\varphi'_l(p+\xi')||.$$

  By Corollary  \ref{cor3}  applied with $\phi=\varphi'_l$  and $\xi=-\xi'$, we deduce 
  $$ | \phi|(K_0+\varepsilon)\gg 
  c(p)|\varphi'_l|.$$
  Whence
  \begin{equation}
  \label{rr1}|\phi'|=|\varphi'_l|\ll  | \phi| \frac{(K_0+\varepsilon)}{
  c(p)} \ll a\frac{(K_0+\varepsilon)}{
  c(p)}. \end{equation}
  
  We define 
\begin{equation*}\begin{split}
[a](y')&=\phi'({{p}})  \,\,\, {\rm{and}} \,\,\,\, y=i_\erre(y')\in \Gamma_p^\g,\\
[a](\upxi')&=\tilde\phi(\xi,\xi')\,\,\, {\rm{and}} \,\,\,\,\upxi=i_\erre(\upxi')\in A^\g,\\
 \end{split}\end{equation*}
Then  \begin{equation*}
\phi({{x}}+y+{\upxi})=0,
\end{equation*} for $y\in \Gamma_p^\g$. We shall still show that $||\xi||\le \varepsilon'$.
By relation (\ref{rr1}),
 $\frac{| \tilde\phi|}{a}\ll \max\left(1,\frac{(K_0+\varepsilon)}{
  c(p)} \right)$ and $||(\xi,\xi')||\le
{\varepsilon}$. We then obtain $$||\upxi||\ll
\frac{| \tilde\phi|}{a}||(\xi,\xi')||\ll
\max\left(1,\frac{(K_0+\varepsilon)}{
  c(p)} \right).$$

We conclude that
$$x\in V_{K_0}\cap (B_{\phi}+\Gamma^\g_p+\mathcal{O}_{\varepsilon'}),$$
with            $||\upxi||\le \varepsilon'\ll \varepsilon\max(1,\frac{(K_0+\varepsilon)}{
  c(p)} )$.

\end{proof}

\section{The Proof of the Main Theorem}
\label{dimmain}
\begin{proof}[Proof of Theorem \ref{main}]

Thanks to Theorem \ref{equii}, to prove Theorem \ref{main}  is sufficient to prove that Conjecture \ref{bofu1} implies Conjecture \ref{nd} ii.
Furthermore, in view of Lemma \ref{stabi}, we can assume that $V$ has finite stabilizer.

 Let $\gamma$ be a set of generators of $\Gamma^{\underline{1}}$ satisfying relation (\ref{gammagrande}) for $\varepsilon\le K_0$.  Note that $\Gamma \subset \Gamma^\g$. Let $\varepsilon \le \min(K_0,\varepsilon_0(\gamma))$ where $\varepsilon_0(\gamma)$ is defined as in Corollary \ref{cor3}.

By Lemma  \ref{peso} ii., for all $\varepsilon\ge0$,
$$ S_{d+1}(V_\cbh, \Gamma_\varepsilon) \subset \Big(V_\cbh \cap \bigcup_{\substack{\phi\,\,\,\,{\rm{weighted}}\\
{\rm{cod \phi}}\ge d+1}}\left(B_\phi+(\Gamma^{\g})_{\varepsilon} \right)\Big) .$$

By Proposition \ref{speciali}, for all $\varepsilon\le K_0$,
$$\Big(V_\cbh \cap \bigcup_{\substack{\phi\,\,\,\,{\rm{weighted}}\\
{\rm{cod \phi}}\ge d+1}}\left(B_\phi+(\Gamma^{\g})_{\varepsilon} \right)\Big)\hookrightarrow  \Big(
(V_\cbh\times \gamma) \cap \bigcup_{\substack{\tilde\phi=(\phi|\phi')\,\,\,\,{\rm{special}}\\
{\rm{cod \tilde\phi}}\ge d+1}}\left(B_{\tilde\phi}+\mathcal{O}_{\varepsilon} \right)\Big).$$

By Theorem \ref{centro}, for $\varepsilon>0$, there exist  a positive real   $M$ such that 
$$ \bigcup_{\substack{\tilde\phi=(\phi|\phi')\,\,\,\,{\rm{special}}\\
{\rm{cod \phi}}\ge d+1}}\left((V_\cbh\times \gamma) \cap\left(B_{\tilde\phi}+\mathcal{O}_{\varepsilon/M }\right)\right) \subset
 \bigcup_{\substack{\tilde\phi=(\phi|\phi')\,\,\,\,{\rm{special}}\\
{\rm{cod \tilde\phi}}\ge d+1\\ |\tilde\phi|\ll M}}\left((V_\cbh\times \gamma) \cap\left(B_{\tilde\phi}+\mathcal{O}_{\varepsilon'/|\tilde\phi|}\right)\right).$$
Note that on the right the union is over finitely many sets, because $|\tilde\phi|\ll M$.

 Let $\equattro$ be as in Corollary \ref{fine}. Choose $\varepsilon'\le \equattro$ (and consequently choose $\varepsilon$). Note that $|\phi|\le |\tilde\phi|$. By Corollary \ref{fine}, 
$$(V_\cbh\times \gamma) \cap \left(B_{\tilde\phi}+\mathcal{O}_{\varepsilon'/|\tilde\phi| }\right)$$
is non Zariski-dense in $V\times \gamma$.

We conclude that $ S_{d+1}(V_\cbh, \Gamma_{\varepsilon/M})$ is embedded in a finite union of non Zariski-dense sets.

\end{proof}

\vskip1cm

\end{document}